\documentclass[12pt]{article}
\usepackage{url}
\usepackage{hyperref}
\usepackage{fullpage}
\usepackage{amsmath}
\usepackage{amssymb}
\usepackage{amsthm}
\usepackage{graphicx}
\usepackage{amssymb}
\newcommand{\as}{\mbox{a.s.}}
\newcommand{\beq}{\begin{eqnarray*}}
\newcommand{\feq}{\end{eqnarray*}}
\newcommand{\beqn}{\begin{eqnarray}}
\newcommand{\feqn}{\end{eqnarray}}
\newcommand{\ch}[1]{\mbox {\bf 1}_{\{#1\}}}
\newcommand{\Z}{{\mathbb {Z}}}
\newcommand{\N}{{\mathbb {N}}}

\newcommand{\hr}[1]{\newline\href{#1}{\url{#1}}}

\newtheorem{theorem}{Theorem}

\newtheorem{lem}{Lemma}
\newtheorem*{remark*}{Remark}

\newcommand{\veps}{\varepsilon}
\title{On a species survival model}
\author{
Iddo Ben-Ari\footnote{Department of Mathematics, University of Connecticut, Storrs, CT 06269, USA; benari@math.uconn.edu}
\and
Anastasios Matzavinos\footnote{Department of Mathematics, Iowa State University, Ames, IA 50011, USA; tasos@iastate.edu}
\and
Alexander Roitershtein\footnote{Department of Mathematics, Iowa State University, Ames, IA 50011, USA; roiterst@iastate.edu}
}
\begin{document}
\maketitle
\abstract{In this paper we provide some sharp asymptotic results for
a stochastic model of species survival recently proposed by Guiol, Marchado, and  Schinazi.}
\section{Introduction and statement of results}
Recently, Guiol, Marchado, and  Schinazi \cite{GMS} proposed a new mathematical framework for modeling
species survival which is closely related to the discrete Bak-Sneppen evolution model. In the original  Bak-Sneppen model \cite{BSM}
a finite number  of species are arranged in a circle, each species being characterized by
its location and a parameter representing the  \emph{fitness} of the species and taking values between zero and one. The number
of species and their location on the circle are fixed and remain unchanged throughout the evolution of the system.
At discrete times $n=0,1,\ldots,$ the species with the
lowest fitness and its two immediate neighbors  update simultaneously  their fitness values at random. The Bak-Sneppen evolution model is
often referred to as an ``ecosystem" because of the local interaction between different species. The distinguishing feature of the
model is the emergence of self-organized criticality \cite{BAKB, evolution, GMW,jensen}  regardless the simplicity and
robustness of the underlying evolution mechanism. The Bak-Sneppen model has attracted significant attention over the past few decades,
but it has also been proven to be difficult for analytical study. See for instance \cite{GMW} for a relatively recent survey of the model.
\par
The asymptotic behavior of the Bak-Sneppen model, as the number of species gets arbitrarily large, was conjectured on the basis of computer
simulations in \cite{BSM}. It appears that the distribution of the fitness is asymptotically uniform on an interval $(f_{\mbox{crit}}, 1)$ for some
critical parameter $f_{\mbox{crit}}$, the value of which is close to $2/3$ \cite{BAKB, jensen}.
\par
Guiol, Marchado, and  Schinazi \cite{GMS} were able to prove a similar
result for a related model with a stochastically growing number of species.
Their analysis is based on a reduction to the study of a certain random walk,
which allows them to build a proof using well-known results
from the theory of random walks. The main result of \cite{GMS} is thus based on
general properties of Markov chains, and  suitable variations of the result can in
principle be carried out to other similar models.
\par
In this paper we  focus on the model introduced in \cite{GMS} as is. Our aim
is to elucidate the underlying mechanism responsible for the phenomenon described in
\cite{GMS} by sharpening the estimates that lead to the major qualitative statement therein.
We proceed with a description of the Guiol, Marchado, and  Schinazi (GMS) model.
\par
In contrast to the Bak-Sneppen model, the number of species in the GMS model is random and changes in time,
and only the species with the lowest fitness is randomly replaced.
The local interaction between species is not considered in the GMS model, and therefore the spatial structure of the population
is of no importance. Let $p>\frac 12$ be given and denote $q=1-p.$ Let $\Z_+$ denote the set of non-negative integers and let
$X=(X_n:n\in \Z_+)$ be a discrete-time birth and death process with the following transition probabilities: from each site, $X_n$ increases by $1$ with probability $p$; from each site different than $0$, $X_n$ decreases by $1$ with probability $q=(1-p)$; finally, at $0$, $X_n$ stays put with probability $q.$  Thus $X$ is a
nearest-neighbor transient random walk on the integer lattice $\Z_+$ with holding times and reflection barrier at zero. Throughout the paper we assume that
$X_0=0$ with probability one.
\par
The model describes the evolution of a population of species: newborn species  are given "equal opportunity" by assigning them a random fitness, and the least fit species are the first to die.  A jump to the right represents birth of a new species, whereas a jump to the left represents death of an existing species. Thus $X_n$ represents the number of {\it species alive} at time $n.$ When a new species is born,  it is assigned a {\it fitness}. The fitness is a uniform $[0,1]$ random variable independent on the fitness of all previously born species as well as of the path of the process $X.$ When $X$ jumps to the left, the species with the lowest fitness is eliminated. We remark
that, in a different context,  a similar model was considered by Liggett and Schinazi in \cite{LS09}.
\par
Fix $f \in (0,1)$. We examine the model by considering two coupled random processes, $L=(L_n:n\in \Z_+)$ (for {\it lower or left})
and $R=(R_n:n\in \Z_+)$ (respectively, for {\it right}),  where $L_n$ denotes the number of species alive at time $n$ whose fitness is less than $f$
while $R_n$ denotes the number of the remaining species alive at time $n$.
\par
Observe that  $L_n$ increases by $1$ if $X_n$ does and the newborn species has fitness less than $f,$ and $L_n$ decreases by $1$ whenever $X_n$ decreases by one and $L_n$ is not zero. The value of $L_n$ remains unchanged when either $X_n$ increases by $1$ and the newborn species has fitness at least $f$ or $X_n$
decreases by $1$ and $L_n=0.$
\par
Notice that when it is not at zero, the process $L$  evolves as a nearest-neighbor random walk with probability $pf$ of jumping to the right, probability $q$ of going
to the left, and probability $1-pf -q $ of staying put. When at zero, $L_n$ jumps to the right with probability  $pf$, and stays put with
the complementary probability $1-pf.$ Thus $L$ is a Markov chain. Since $p>q,$ it is positive recurrent if $pf < q,$ null-recurrent if $pf =q,$ and is otherwise transient. In what follows we will denote the critical value $q/p$ of the parameter $f$ by $f_c.$    
\par
Consider the process $B=(B_n:n\in \Z_+),$ where $B_n$ is the total number of species born by time $n$ with fitness at least $f$. Observe that $B$ is a non-decreasing Markov chain (formed by sums of i.i.d. Bernulli variables), which jumps one step up with probabilities $p(1-f)=p-q$ and stays put with the complementary probability $2q.$ It is shown in \cite{GMS} that when $f=f_c,$ we have that for any $\veps >0$
\beqn
\label{icor}
0 \le \limsup\limits_{n\to\infty} \frac{B_n - R_n}{ n^{1/2+\veps} } \le \frac{2}{q}, \qquad \as,
\feqn
while each species with fitness less than the critical value disappears after a finite (random) time. It follows from \eqref{icor} that 
the distribution of species with fitness higher than $f_c$ approaches a uniform law. 
\par
We sharpen the above result by proving the following theorem. Recall that $f_c=q/p.$ 
\begin{theorem}
\label{th:limit}
Suppose that $f=f_c.$ Then
\begin{enumerate}
\item
$\displaystyle \limsup_{n\to\infty} \frac{B_n -R_n}{\sqrt{4q n \ln \ln n}} = 1,\qquad \mbox{\rm a.s.}$
\item
$\displaystyle  \frac{B_n -R_n}{\sqrt{2qn}} \Rightarrow|N(0,1)|,$ where $N(0,1)$ denotes
a mean-zero Gaussian random variable with variance one, and $ \Rightarrow$ stands for convergence in distribution.
\item $(L_n:n\in \N)$ is a recurrent Markov chain (visiting zero infinitely often) and hence each species with fitness less than the critical 
value becomes eventually extinct with probability one.
\end{enumerate}
\end{theorem}
The proof of the theorem is included in Section~\ref{proofs}.  Notice that by the law of large numbers $B_n\sim n(1-f_c),$ $\as,$ as $n\to \infty.$ Moreover, the fitness of any randomly chosen species from $B_n$ is distributed uniformly on $(f_c,1).$ Hence the theorem implies that
this uniform law on $(f_c,1)$ is the limiting distribution for the fitness of a randomly chosen species being alive at time $n$ and having fitness larger than $f_c.$
\section{Proof of Theorem~\ref{th:limit}}
\label{proofs}
There is no loss of generality assuming that $X$ is obtained recursively from an i.i.d. sequence of Bernoulli random variables $J=(J_n:n\in \Z_+)$ with
$P(J_n=1)=q,$ $P(J_n=0)=p,$ as follows: $X_0=0$ and for $n\in \Z_+$ we have
\beqn
\label{eq:Xdefn}
X_{n+1} =X_n +  (1-J_n) - J_n (1-s_n),\mbox{ where } s_n:= \ch{X_n=0}.
\feqn
Let $G=(G_n:n\in\Z_+)$ be a Markov chain formed by the pairs $G_n = (L_n,J_n).$ Figure \ref{fig:Gprob} illustrates the transition mechanism of
$G.$
\begin{figure}[t]
\includegraphics[width=\textwidth]{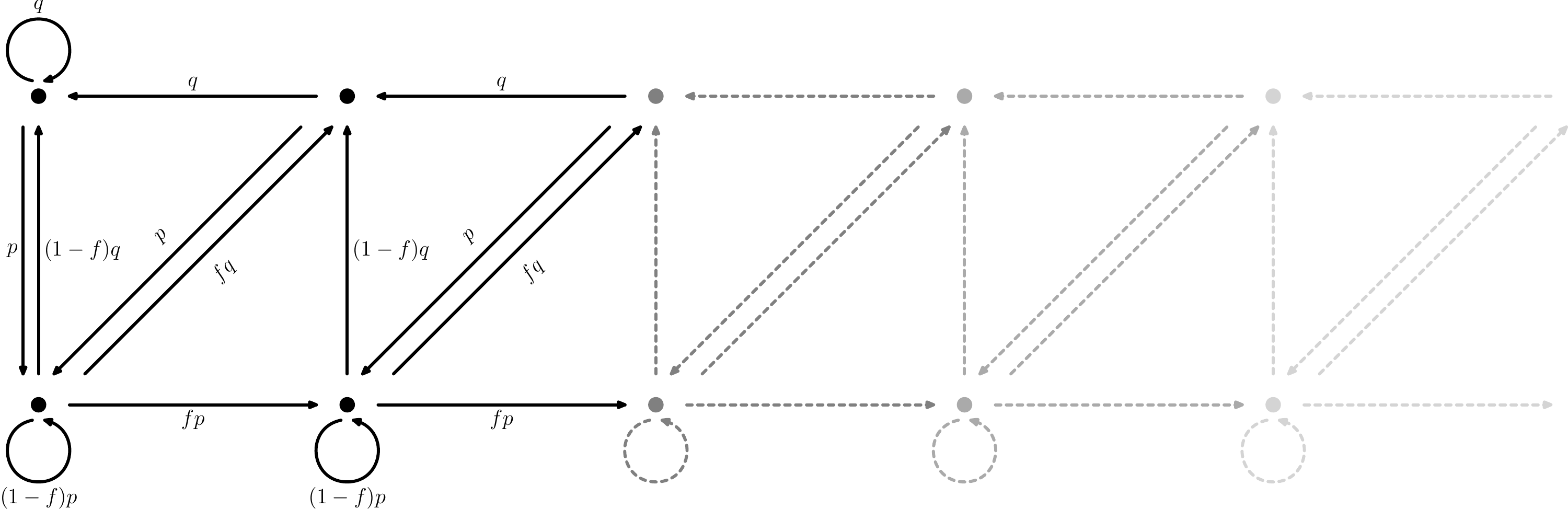}
\caption{Transition probabilities for $G$}
\label{fig:Gprob}
\end{figure}
\subsection{Reduction from $B_n-R_n$ to an occupation time of $G$}
Let $\Delta=(\Delta_n:n\in\Z_+)$ be the process defined through  $\Delta_n= B_n -R_n.$
Notice that $\Delta_n$ increases by $1$ if  and only if  $R_n$ decreases by $1,$ and otherwise stays put. That is
\beq
\Delta_{n+1}-\Delta_n = \ch{X_{n+1}-X_n=-1}  \ch{L_n=0} .
\feq
Thus by \eqref{eq:Xdefn}, $\Delta_{n+1} -\Delta_n = (1-s_n) J_n  \ch{L_n=0} $, and we have
\beq
\Delta_n = \sum_{i=0}^{n-1} (1-s_i) J_i \ch{L_i=0} = \eta_n - \sum_{i=0}^{n-1} J_i s_i,\mbox{ where }  \eta_n:= \sum_{i=0}^{n-1} J_i \ch{L_i=0}.
\feq
Observe that $\eta_n$ is the occupation time (number of visits) of $G$ at $(0,1)$ up to time $n-1$ and that $\sum_{i=0}^{n-1} s_i J_i$
is bounded above by $\sum_{i=0}^\infty s_i.$ Since $X$ is transient and $X_0=0,$ the latter series can be represented
in the form $\sum_{i=0}^\infty s_i=\sum_{i=1}^g h_i,$
where $g$ is the total number of visits to zero (distributed geometrically, according to $P(g=k)=f_c^{k-1}(1-f_c),$ $k\in\N,$ 
see for instance \cite[p.~274]{durrett}) and $h_i$ are successive holding times of $X$
at zero (i.i.d. sequence of geometric random variables,
independent of $g,$ with $P(h_i=k)=q^{k-1}p,$ $k\in\N$). In particular,
\beq
E\Bigl(\sum_{i=0}^\infty s_i J_i\Bigr)\leq E(g)\cdot E(h_1)=\frac{1}{p(1-f_c)}=\frac{1}{2p-1}<\infty.
\feq
Since $L$ and consequently $G$ are recurrent (and thus
$\Delta_n$ is a non-decreasing sequence converging to $+\infty$ with probability one), it therefore suffices to show that Theorem~\ref{th:limit} holds with
$\Delta_n=B_n-R_n$ replaced by $\eta_n$ in its statement.
\paragraph{Excursion decomposition for the path of $L.$}
We decompose the path of $L$ into a sequence of successive excursions from $0,$
each one lasting until (but not including) the first time $L$ returns to $0$ from $1.$
Observe that once the Markov chain $L$ is at zero, it will stay put until $J_k=0$ and the fitness of the newborn particle
is less than $f_c.$ Therefore, the holding time of $L$ at zero during one excursion is a geometric random variable with parameter
$pf_c=q.$ Let $\tilde h_k$ be the i.i.d. sequence of holding times at zero during successive excursions of $L$ from zero.  That is,
$P(\tilde h_k=n)=p^{n-1}q,$ $n \in \N.$ In what follows we will use the notation $\mbox{Geom}(a)$ for the geometric distribution
with parameter $a$ (for instance, we could write $\tilde h_k \sim \mbox{Geom}(q)$). Let
$\sigma_k$ be the $k$-th excursion starting time, that is $\sigma_1=0$ and
\beq
\sigma_k=\inf\{n>\sigma_{k-1}: X_{n-1}=1, X_n=0\},\quad k=2,3,\ldots
\feq
Define
\beq
\mu_k=\sum_{i=\sigma_k}^{\sigma_{k+1}-1} J_i\ch{L_i=0}=\sum_{i=\sigma_k}^{\sigma_k+\tilde h_k-1} J_i\ch{L_i=0}.
\feq
Notice that $(\mu_k:k\in \N)$ is an i.i.d. sequence and
\beqn
\label{law-ln}
\sum_{k=1}^{N_n} \mu_k < \eta_n  \le \sum_{k=1}^{N_n+1} \mu_k .
\feqn
We now compute $\mu:=E(\mu_k)$ using the fact that the sequence of pairs $G_n=(J_n,L_n)$ forms
a Markov chain, the transition mechanism of which is illustrated in Fig.~\ref{fig:Gprob}.  The value of
$\mu$ is equal to the expected number of visits by this Markov chain to the state $(0,1)$ during the period of time
starting at the state $(0,1)$ with probability $q$ and at $(0,0)$ with probability $p,$
and lasting until $G$ leaves the set $\{(0,0), (0,1)\}.$ We thus have, using first step analysis,
\beq
\mu=q(1+\mu)+p(1-f_c)\mu+pf_c\cdot 0=q+\mu(1-pf_c)=q+\mu p.
\feq
Hence $\mu=1.$ Consequently, using \eqref{law-ln} and the law of large numbers, we obtain
\beqn
\label{eq:etan}
\eta_n \sim N_n ~\mbox{as}~n\to \infty,\quad \as
\feqn
\paragraph{Reduction to a simple random walk}
The duration of the $k$-th excursion of $L$ from zero is equal to the sum $\tilde h_k+\alpha_k$ of the holding time at zero $\tilde h_k$ distributed as
$\mbox{Geom}(q),$ and the time $\alpha_k$ that it takes for $L$ to return from 1 to zero. Notice that
\beq
\alpha_k=\sigma_k-(\sigma_{k-1}+\tilde h_k-1).
\feq
For $m\in\N$ let $V_m = \sum_{k=1}^m (\tilde h_k + \alpha_k)$ be the total duration of the first $m$ excursions of $L$ from $0.$
Then $N_k=\max\{m:V_m\le k\}.$ With each excursion we can associate a \emph{skeleton}, which is the path obtained from the excursion
by omitting all transitions from a state to itself, for all states. If we let $\tau_k$ denote the length of the skeleton, then due to the choice of $f_c$ it follows that $\tau_k$ has the same distribution as the time required for the simple symmetric random walk on $\Z$ to get back to $0$ starting from $0.$ Furthermore, the time spent by $L$ at each visit to a site is a geometric random variable, $\mbox{Geom}(2q)$ for sites different than $0$ and $\mbox{Geom}(q)$ for $0.$ Therefore, the length of a single excursion itself is a sum of one
$\mbox{Geom}(q)$ random variable plus a sum of $\tau_k-1$ independent
$\mbox{Geom}(2q)$ random variables. If we replace $\tilde h_k$ with a $\mbox{Geom}(2q)$, then the resulting modified ``excursion time" becomes
a sum of $\tau_k $ copies of a $\mbox{Geom}(2q)$ random variable. Let $V_m'$ denote the total length of the first $m$ excursions modified in this way.
Let $N_k'$ denote the number of such excursions occurred by time $k$, that is
$N_k'= \max\{m:V_m' \le k\}.$ Then $V_m$ stochastically dominates $V_m'$. By the law of large numbers,
\beqn
\label{eq:lln}
\lim_{m\to\infty}  \frac{ V_m - V_m'}{m}  = E (\tilde h_1) - E (h_1') = \frac{1}{2q},
\feqn
where $h_1'$ is a geometric random variable with parameter $2q.$  Letting $T_m = \sum_{k=1}^m \tau_k,$ we obtain
\beq
V_m' = \sum_{k=1}^{T_m} h_k'',
\feq
where $(h_k'':k\in\N)$ is an i.i.d. sequence of random variables, each one distributed as $\mbox{Geom}(2q).$ Thus, by the law of large numbers,
\beqn
\label{eq:Vl'}
V_m' \sim \frac{T_m}{2q}~\mbox{as}~m\to\infty, \quad \as
\feqn
Notice that $T_m$ is distributed the same as the total length of the first $m$ excursions from zero of a simple symmetric random walk.
\subsection{Completion of the proof: CLT and LIL for $\eta_n$}
\paragraph{LIL for $\eta_n$}
We need the following result. Although the claim is a ``folk fact", we give a short proof
for the sake of completeness.
\begin{lem}
\label{lem:T}
\beq
\liminf_{m\to\infty} \frac{T_m}{m^2 /(2 \ln \ln m)}= 1 \mbox{ \rm a.s. }
\feq
\end{lem}
\begin{proof}[Proof of Lemma \ref{lem:T}]
Let $S=(S_n:n\in\Z_+)$ denote the simple symmetric random walk on $\Z$. Let $\gamma_0=0$ and define inductively
$\gamma_{m+1}= \inf \{k> \gamma_m: S_k=m+1\}$. Let $Y_m= \gamma_{m+1}-\gamma_m$ with the usual convention that
the infimum over an empty set is $+\infty.$ Let $\phi(x) = \sqrt{2x \ln \ln x}$ for $x>0.$
By the law of iterated logarithm for $S,$
\beq
\limsup_{n\to\infty} \frac{S_n}{\phi(n) } = \limsup_{n\to\infty} \frac{S_{\gamma_n}}{\phi(\gamma_n)}=\limsup_{n\to\infty} \frac{n}{\phi(\gamma_n)}.
\feq
Since $\phi^{-1}(x)\sim x^2/(2 \ln \ln x)$ as $x \to\infty,$ we obtain
\beq
\liminf_{n\to\infty} \frac{\gamma_n}{n^2 / (2\ln \ln n)} =1.
\feq
Observe that $\gamma_n=\sum_{i=0}^{n-1} Y_i$. The distribution of $t_1$, the time to return to $0$ starting from $0$ for $S$, is equal to the distribution
of $t_1'$, the time to return to $0$ starting from $0$ for the reflected random walk $|S|.$ Since $t_1'$ is equal in distribution to $1+Y_1,$
\beq
\liminf_{m\to\infty} \frac{T_m}{m^2 /(2 \ln \ln m)}= \liminf_{n\to\infty}  \frac{\gamma_n+n}{n^2 / (2\ln \ln n)} =1,
\feq
completing the proof of the lemma.
\end{proof}
Using the lemma along with \eqref{eq:lln} and \eqref{eq:Vl'}, we obtain
\beq
\liminf_{m\to\infty} \frac{2q  V_m}{m^2 /(2\ln \ln m)}=\liminf_{m\to\infty} \frac{2q V_m'}{m^2/(2\ln \ln m)}=1, \quad \as
\feq
Consequently, since $N$ is the inverse of sequence of $V$, we obtain
\beq
\limsup_{k\to\infty} \frac{N_k}{\sqrt{4q k \ln \ln k}}=1, \quad \as
\feq
Combining this with \eqref{eq:etan} completes the proof of the law of iterated logarithm for $\eta_n.$
\paragraph{CLT for $\eta_n$}
We now turn to the proof of the central limit theorem. It is well known (see for instance \cite[p.~394]{durrett}) that
\beq
\lim_{m\to\infty} E \bigl(e^{-\theta T_m/m^2}\bigr) =E \bigl(e^{-\sqrt{2\theta}}\bigr),\qquad \theta \geq 0.
\feq
Therefore, it follows from \eqref{eq:lln} and \eqref{eq:Vl'} that
\beq
\lim_{m\to\infty} E \bigl(e^{-\theta V_m/m^2}\bigr) = \lim_{m\to\infty} E (e^{-\theta V_m'/m^2}\bigr) =
E \bigl(e^{-\sqrt{\theta/q}}\bigr).
\feq
The function $\theta \to e^{-c \sqrt{2  \theta}},$  $\theta \geq 0,$ is the Laplace transform of a positive stable law with index $1/2$
whose density function is given by (see for instance \cite[p.~395]{durrett})
\beq
\varphi_c(u) = \ch{u \geq 0} \frac{c e^{-c^2/2u} }{\sqrt{2\pi u^3}}.
\feq
We intend to use this formula with $c=\frac{1}{\sqrt{2q}}.$ Set $V_0=0$ and observe that for all $k\in \N,$
\beq
P(N_k \le u  ) = P_0 (V_{\lfloor u\rfloor} \ge  k).
\feq
Fix $s>0$ and let $u=\sqrt{k}s$. Then
\beq
P(N_k \le \sqrt{k}s) =P(V_{\lfloor \sqrt{k}s \rfloor}\ge k ) =
P\Bigl (\frac{V_{\lfloor \sqrt{k}s \rfloor}}{(\lfloor \sqrt{k} s\rfloor)^2} \ge \frac{1}{s^2}(1+o(1)) \Bigr) \underset{k\to\infty}{\to}
\int_{s^{-2}}^\infty \varphi_c (u) du.
\feq
Differentiating the right-hand side we obtain
\beq
\lim_{k\to\infty}  P(N_k \le \sqrt{k}s) =  \int_0^s \frac{2 e^{- u^2/(2c^{-2})}}{\sqrt{2\pi c^{-2}}}du.
\feq
Therefore $N_k /\sqrt{k}$ converges weakly to the absolute value of a centered normal random variable with variance equal to $c^{-2} = 2q$. Combining this with \eqref{eq:etan} completes the proof of the central limit theorem for $\eta_n.$ \qed

\end{document}